\pdfoutput=1
\documentclass[reqno]{amsart}
\usepackage[utf8]{inputenc}
\usepackage{amsmath,amssymb,amsthm,amsfonts,amstext}

\usepackage[colorlinks=true,linkcolor=blue,citecolor=magenta]{hyperref}

\usepackage{tikz} 
\usetikzlibrary{arrows,positioning}

\let\ID\inlinedisplay
 
\def\<{\langle}
\def\>{\rangle}
\def\chv#1{\left\{\,#1\,\right\}}

\def\abs#1{\left\vert #1\right\vert}
\def\norm#1{\left\Vert #1\right\Vert}
\def\prt#1{\left( #1\right)} 
\def\crt#1{\left[ #1\right]} 

\def\piso#1{\left\lfloor \;#1\; \right\rfloor}

\newcommand{\mc}[1]{{\mathcal #1}}

\newcommand{\bb}[1]{{\mathbb #1}}

\newtheorem{theorem}{Theorem}
\newtheorem{prop}{Proposition}

\newtheorem{lemma}{Lemma}

\newtheorem*{teoutro}{Theorem} 
\newtheorem*{prop*}{Proposition}

\newtheorem*{lema*}{Lemma}

\author[Bernardo Costa]{Bernardo F.P. Costa}
\address{\noindent Instituto de Matemática - UFRJ -  Centro de Tecnologia - Bloco C\\
  Av. Athos da Silveira Ramos 149 \\
   Cidade Universitária, Rio de Janeiro - RJ - Brazil}
\email{bernardofpc@im.ufrj.br}

\author{Conrado Costa}
\address{\noindent Mathematical Institute of Leiden University \\
Niels Bohrweg 1, \\
2332 CA, Leiden, Netherlands}
\email{conrado@impa.br}

\author{Milton Jara}
\address{\noindent Instituto de Matem\'atica Pura e Aplicada\\ Estrada Dona Castorina 110\\ 22460-320 Rio De Janeiro, Brazil.}
\email{mjara@impa.br}

\title{Reaction-Diffusion models: From Particle Systems to SDE's}
\date{\today}


\begin{document}

\begin{abstract}
In this article we present a construction of a family of Reaction-Diffusion models
that converge after scaling to the solution of the following SDE:
\begin{equation*}
  \begin{cases}
    d\zeta_t(x) = \Big[\Delta_V\zeta_t(x) - \beta\prt{\zeta_t(x)}^k\Big] dt + \sqrt{\alpha\prt{\zeta_t(x)}^\ell} \, dB^x_t \quad \forall x \in V\\
    \zeta_0(x) = \rho_0(x)
  \end{cases}
\end{equation*}
where $\alpha,\beta>0$ and $k$, $\ell$ are positive integers.
\end{abstract}

\maketitle


\section{Introduction}

In this paper we explore the use of Reaction-Diffusion models in particles systems
to derive solutions to a family of Stochastic Differential Equations (SDE).
Start with a one-dimensional Ordinary Differential Equation (ODE):
\begin{equation}
  \label{eq:LLN}
  d\zeta_t = F(\zeta_t).
\end{equation}
It is possible to construct a family of particle systems,
with birth and death transitions,
whose densities converge in uniform norm
to the solution of the ODE~\eqref{eq:LLN}.
This type of result is known as a ``Functional Law of Large Numbers''.

Due to the probabilistic nature of the particle systems,
one second step is to consider the fluctuations of the densities around the limit trajectory,
in the spirit of Functional Central Limit Theorems.
Under reasonable assumptions on $F$,
and a proper scaling of time and space,
the fluctuations around a fixed point of equation~\eqref{eq:LLN} are solutions of the following SDE:
\begin{equation}
  \label{eq:CLT}
  d\zeta_t = -\beta(\zeta_t)^k + \sqrt{\alpha\prt{\zeta_t}^{\ell}}\, dB_t
\end{equation}
where $\alpha, \beta >0$ and $k,\ell$ are positive integers satisfying $k \geq \ell$.

Beyond the one-dimensional case, we can apply the same techniques using
particle systems on a finite set $V$ to derive solutions to vector-valued SDE's.
If the number of particles on each site of $V$ evolves independently,
the finite-dimensional case is simply the product of the one-dimensional solutions.
But if we allow particles to jump from any site $x\in V$
to any other site $y\in V$ following a transition kernel $p(x,y)$,
and if we properly scale the jump rates,
the scaling limit of our particle system now solves:
\begin{equation}
  \label{eq:kl}
  d\zeta_t(x) = \prt{\Delta_{V,p} \zeta (x)-\beta(\zeta_t(x))^k}dt  + \sqrt{\alpha\prt{\zeta_t(x)}^{\ell}} \, dB_t^x \qquad \forall x \in V.
\end{equation}
This is a finite-dimensional stochastic reaction-diffusion equation
with diffusion given by $\Delta_{V,p}\zeta(x) = \sum_{y} \prt{p(y,x) \zeta(y) - p(x,y) \zeta(x)}$
and noise given by $\prt{B^x_\cdot}_{x \in V}$ a $|V|$-dimensional Brownian Motion.

This result was established in \cite{Cos},
where the diffusion term was given by the discrete Laplacian on the graph $(V,E)$:
$\Delta_{V}\zeta(x) = \sum_{y\sim x}  \zeta(y) - \zeta(x)$.

\medskip
 
The introduction of properly rescaled jump rates to allow non-trivial solutions
in the finite-dimensional case invites one to change perspective
to a more abstract setting.
Instead of considering space-time scalings to derive solutions to SDE's,
we now tailor-make the jump, birth and death rates
so as to conveniently conserve macroscopic quantities.
This more abstract perspective allows us to derive
solutions to SDE's such as~\eqref{eq:kl}, including the case when $k < l$:
\begin{theorem}
  \label{T:1}
  Given positive reals $\alpha$ and~$\beta$,
  positive integers $k$ and~$\ell$,
  and a transition kernel $p$ on $V \times V$,
  there are Reaction-Diffusion models that generate,
  for every density $\rho_0$ on $\bb{R}^V$,
  a family of stochastic processes $\prt{\eta^n_\cdot}_n$,
  that converges after scaling to the (unique) solution
  $\zeta^*_\cdot$ of the stochastic differential equation~\eqref{eq:kl}
  with initial condition $\zeta_\cdot = \rho_0$.
\end{theorem}

Some particular cases of equation \eqref{eq:kl} are of special interest,
due to their connection to stochastic partial differential equations.
The case $\beta =0$, $\ell =2$ corresponds to a {\em parabolic Anderson model} with space-time
noise. In the parabolic Anderson model, $\zeta_t(x)$ corresponds to the expected
number of particles of a system of independent random walks which split with rate
$(dB_t^x)^+$ and annihilated with rate $(dB_t^x)^-$. Although these rates are not
well defined, they make sense in a weak sense using It\^o's integral, see
\cite{ErhdHolMai} for a recent reference. In that article, equation \eqref{eq:kl}
is interpreted as a {\em catalytic} equation: the noises $B_t^x$ act as an
independent catalyst which favours of penalizes the reproduction of the
individuals. In our case, we obtain the same equation as the scaling limit around
an equilibrium point of a {\em self-catalytic} system, on which the reproduction
of individuals is changed by the presence or absence of other individuals. In the
one-dimensional lattice, this model has striking combinatorial structure
\cite{BorCor}, which allows to study various correlation functions in terms of
contour integrals.

The case $\beta =0$, $\ell=1$ corresponds to a discrete version of {\em Dynkin's super-Brownian motion}, see \cite{LeGall}. On these two cases, the process \eqref{eq:kl} appears as the random density of an underlying Markov process. From the point of view of the models discussed here, in both situations the case $\beta > 0$, $k =\ell$ seems to be natural. It remains an open question to understand these case in terms of densities of associated models.

\bigskip

This paper will proceed in three steps.
We first present, in the one-dimensional case,
the main objects we will be dealing with,
and the scalings required to obtain~\eqref{eq:CLT}.
Then, we outline the strategy of the proof on the general case,
reviewing the martingale approach of Strook and Varadhan.
Finally, the remaining sections complete the details of our argument.

\section{Overview for the one-dimensional case}
\label{sec:1d}

Here, we focus on the one-dimensional case
to concentrate on the scaling requirement of the birth and death rates
in order to observe fluctuations of the solutions.

\subsection{Physical background}
Reaction-diffusion equations have been used to model a large class of
real-world phenomena,
such as population dynamics~\cite{CanCos} or
pattern formation and morphogenesis~\cite{V}.
From a probabilistic setting,
there are convergence results for models of chemical reactions~\cite{Bl92,AT},
and more particularly in the particle system setting,
for tunnelling and metastability~\cite{FarLanTsu,LanTsu}.

The physical image we have in mind for scaling limits is as follows:
we consider that reaction is taking place on a given region
and we are able to detect ``particles'' with mass $1/m$.
At level $m$, we detect $\piso{m\zeta}$ particles with our instrument,
and observe creation and annihilation also in this scale.
The net change of mass on a small time interval
is a function of the total mass in the region (say $\zeta$),
and is independent of the measuring device.
So, if (at precision $m$) we observe $i$ particles,
the transition rates should be given by:
\begin{equation}\label{rates}
r_m(i,i+1) = mF_+\big( \tfrac{i}{m}), \quad r_m(i,i-1) = m F_-\big(\tfrac{i}{m}\big).
\end{equation}

In this way, the total mass change on the region is given by
$F(x) = F_+(x) -  F_-(x)$,
where  $F_+,F_-$ are the functions ruling creation and annihilation of particles
when evaluated at the total mass.

\subsection{Functional Law of Large Numbers}
Now, take $F$ a $C^\infty(\bb{R}_+)$ function,
and decompose it as $F = F_+ - F_-$, where $F_+, F_-$ are non-decreasing $C^\infty$ functions.
The modeling above indicates a way to approximate solutions of the ODE~\eqref{eq:LLN}
by means of a stochastic process.
Interpreting the rates in~\eqref{rates} as creation and annihilation of particles,
we obtain birth-and-death chains $\{\eta^m_t; t \geq 0\}_{m \in \bb N}$ in $\bb N$.
The continuous-time rates $F_+$ and $F_-$
are known as the {\em reaction} rates.

In order to avoid explosions in finite time,
we impose a growth condition on $F_+$;
for simplicity we assume that $F_+(\zeta) \leq C(1+\zeta)$ for some finite constant $C$.
With these definitions, it turns out that $\frac{\eta^m_t}{m}$
converges to the solutions of \eqref{eq:LLN}.
Let's sketch the argument.

First we consider the generator:
\begin{equation}
\label{eq:L0}
L_mf(\eta) =  mF_+\big( \tfrac{\eta}{m}) \crt{f(\eta + 1) - f(\eta)} +  m F_-\big(\tfrac{\eta}{m}\big)\crt{f(\eta -1) -f(\eta)}.
\end{equation}
and apply it to $f_m(\eta) = \tfrac{\eta}{m} =: \zeta^m$.
This yields $L_m\zeta^m = F_+(\zeta^m) - F_-(\zeta^m) = F(\zeta^m)$,
where we slightly abuse notation writing $L_m\zeta^m$ instead of $[L_mf_m](\eta^m)$.
The Dynkin martingale for $\zeta^m$ and $L_m$ is (see~\cite[Theorem 3.32]{Lig}):
\begin{equation}
  \label{eq:dm_1}
  M^m_t = \zeta^m_t - \zeta^m_0 - \int_0^t F(\zeta^m_s)\, ds.
\end{equation}
and the variance of $M^m_t$ is given by $\<M^m\>_t = \int_0^t Q_m(\zeta^m_s)\, ds$,
where in general
\[
Q_mf(\eta) = mF_+\big( \tfrac{\eta}{m}) \crt{f(\eta + 1) - f(\eta)}^2 +  m F_-\big(\tfrac{\eta}{m}\big)\crt{f(\eta -1) -f(\eta)}^2.
\]
Here, $Q_m(\zeta^m_s) = \frac{1}{m} \prt{F_+(\zeta^m) + F_-(\zeta^m)}$.

\medskip

Then, we prove tightness of $\zeta^m$,
which allows us to consider convergent subsequences of $\zeta^m$.
Take one converging in the uniform topology to $\zeta^*$,
a random variable on $D([0,\infty), \bb R)$ with continuous paths.
Since $Q_m(\zeta^{m}_s) \to 0$, the martingales $M^m_s$ converge to zero uniformly in compact time intervals.
Together with $F(\zeta^{m}_s)\to F(\zeta^*_s)$ also uniformly,
taking limits on the Dynkin formula~\eqref{eq:dm_1} we obtain:
\[
  \zeta^*_t = \zeta^*_0 + \int_0^t F(\zeta^*_s)\, ds
\]

Therefore $\zeta^*$ solves~\eqref{eq:LLN},
and since this equation admits a unique solution $\phi(\cdot)$ starting from $x_0$,
if we take initial conditions such that $\zeta^m_0 \to x_0$
then the family of probability measures induced by $\zeta^m$ converge to $\delta_\phi$.
This means that we can approximate solutions to~\eqref{eq:LLN} using a birth and death chain.

\subsection{Fluctuations}
Now, if $F$ has a zero at $x_0$, then $\phi(t) = x_0$ for any $t \geq 0$,
and this limit theorem is not very informative.
It makes sense in that case to perform a different scaling on $\{\eta^m_t; t \geq 0\}$
in order to capture a more informative limiting behaviour.
Without loss of generality, we can assume $x_0=0$.
The quantity of interest now is $\zeta_t^m = \eta^m_{tm^b}/m^a$,
where $a$ and $b$ are parameters to be properly set.

To understand the limiting behaviour of $\zeta_t^m$,
we study the behaviour around $\zeta=0$ of the smooth functions
$F(\zeta) = F_+(\zeta) - F_-(\zeta)$ and
$G(\zeta) = F_+(\zeta) + F_-(\zeta)$.
This is determined by their first non-zero derivative at $\zeta = 0$.
Let $k$, $\ell$ be such that
\[
F^{(k)}(0) \neq 0 \text{ and } F^{(i)}(0)=0 \text{ for any } i <k,
\]
\[
G^{(\ell)}(0) \neq 0 \text{ and } G^{(i)}(0)=0 \text{ for any } i <\ell.
\]
Since $F_+$ and $F_-$ are non-decreasing functions, we have $k \geq \ell$.

From the accelerated processes $\eta^m_{tm^b}$, we get a new family of generators $\mc{L}_m$,
satisfying $\mc{L}_m f(\eta) = m^b L_m f(\eta)$.
Consider the functions $f_m(\eta) = \tfrac{\eta}{m^a} =\zeta^m$ and compute $\mc{L}_m \zeta^m$
\begin{align*}
  \mc{L}_m\zeta^m & = m^{b+1 -a} \crt{F_+(m^{a-1}\zeta^m) -F_-(m^{a-1}\zeta^m)} = m^{b+1 -a}F(m^{a-1}\zeta^m) \\
                  & = m^{b+1 -a} \crt{-\beta(m^{a-1}\zeta^m)^k + m^{(a-1)k+1}O((\zeta^m)^{k+1})}
\end{align*}
We still have a Dynkin martingale for $\zeta^m$ and $\mc{L}_m$:
\begin{equation}
  \label{eq:dm_rescaled}
  M^m_t = \zeta^m_t - \zeta^m_0 - \int_0^t \mc{L}_m\zeta^m_s\, ds
\end{equation}
with variance given by:
\[
\<M^m\>_t = \int_0^t \mc{Q}_m(\zeta^m_s)\, ds
\]
where
\begin{align*}
  \mc{Q}_m\zeta^m & = m^{b+1 -2a} \crt{F_+(m^{a-1}\zeta^m) + F_-\big(m^{a-1}\zeta^m)} = m^{b+1 - 2a}G(m^{a-1}\zeta^m) \\
                  & = m^{b+1 -2a} \crt{\alpha(m^{a-1}\zeta^m)^\ell + m^{(a-1)\ell+1}O((\zeta^m)^{l+1})}.
\end{align*}

To obtain non trivial terms in the integrands above, we need to set $a$ and~$b$ such that
\begin{align*}
  b+1 -  a -    k(1-a) & = 0 \\
  b+1 - 2a - \ell(1-a) & = 0.
\end{align*}
which in turn give $a = 1-\frac{1}{1+k-\ell}$ and $b = \frac{k-1}{1+k\ell}$.
Note that since $a<1$, we are zooming-in $\frac{\eta^m_t}{m}$;
and since $a>0$ we can still expect to observe a continuous process in the limit.
Under these choices of parameters, we see that $\mc{L}_m\zeta^m \to -\beta(\zeta^m)^k$
and $\mc{Q}_m\zeta^m \to \alpha(\zeta^m)^l$, uniformly if $\zeta$ is bounded.
As previously, we use tightness and pass to a convergent subsequence.

Now the Dynkin martingales $M^m_t$ do not vanish in the limit.
However, if we prove that $M^m$ is uniformly integrable,
the convergence of $\zeta^m$ to $\zeta^*$ implies the convergence of
$M^m$ to $M^*$, a martingale with quadratic variation given by
\[
  \<M^*\>_t = \int_0^t \alpha (\zeta^*_s)^l\,ds.
\]
By the Martingale Representation Theorem
(see discussion on section~\ref{sec:strategyofproof})
we can write on a convenient probability space that
$M^*_t = \int_0^t\sqrt{\alpha (\zeta^*_s)^l} \, dB_s$
where $B_\cdot$ is a Brownian motion.
Now, we can finally take limits on~\eqref{eq:dm_rescaled} and obtain:
\[
  \zeta^*_t = \zeta^*_0 + \int_0^t -\beta(\zeta^*_s)^k\, ds + \int_0^t\sqrt{\alpha (\zeta^*_s)^l} \, dB_s.
\]

If we take initial conditions $\eta^m_0$ such that $\zeta^m_0 \to \rho_0$,
$\zeta^*$ solves \eqref{eq:CLT} with initial condition $\rho_0$.
Since this equation corresponds to a unique law $P^*$
(depending also on $\alpha$, $\beta$, $k$ and~$l$)
the family of probability measures induced by $\zeta^m$ converge to $P^*$.
This means that we can approximate solutions to~\eqref{eq:CLT}
after properly scaling time and space on a birth and death chain.

\section{Strategy of the proof}
\label{sec:strategyofproof}

The proof of our main result follows the Stroock-Varadhan martingale method
to prove convergence of Markov processes to diffusions.
Although very popular in scaling limits of interacting particle systems (see~\cite{KL,KomLanOll}),
the martingale method is less employed in the context of heavy-traffic limits of queuing networks,
see~\cite{PanTalWhi} for a more detailed discussion.

First, we construct in subsection~\ref{ssec:construction}
a family of particle systems $\eta^n$ and from it
the corresponding scaled processes $\zeta^n$ with initial condition $\zeta^n_0$.
On the one hand, we depart from the physical point of view of the previous section
in favor of an approach that builds a particle system
so as to obtain in the limit the desired convergence to a given SDE.
On the other hand, the structure of the proof remains the same,
and we devote this section to presenting in more detail
how scaling limits of particle systems
give rise to solutions of Martingale Problems,
and how those correspond to solutions of SDE's.

\medskip

Say we have a family of processes $\zeta^n_\cdot$, issued from particle systems,
with infinitesimal generators $\mc{L}_n$
such that for every $f \in C^2(\bb{R}^d)$ and $\zeta \in \bb{R}^d$
\[
  \mc{L}_n f (\zeta) \to \mc{L}_*f(\zeta)
\]
uniformly in compact sets,
where $\mc{L}_*$ is a second-order elliptic differential operator.

Assume for simplicity that $\zeta^n_\cdot$ converges to $\zeta^*_\cdot$ in law.
So, for $f \in C^2(\bb{R}^d)$
the Dynkin martingales $M^{f,\mc{L}_n}_t$ converge to a limit process $M^{f,*}_t$:
\begin{align*}
         M^{f,*}_t & = f(\zeta^*_t) -  f(\zeta^*_0) - \int_0^t \mc{L}_* f(\zeta^*_s) \, ds.
\end{align*}
To establish that $\zeta^*_\cdot$ is a solution to a Martingale Problem associated to $\mc{L}_*$,
we need to show that $M^{f,*}_t$ are indeed local martingales
for the coordinate functions $f_i(\zeta) = \zeta(i)$
and $f_{i,j}(\zeta) = \zeta(i)\zeta(j)$ ~\cite[pp. 315--316]{KS}.

Therefore we introduce stopping times
$\tau^n_A$ and~$\tau^*_A$ such that $\tau^n_A \to \tau^*_A$
and the family $M^{f_i,\mc{L}_n}_{t\wedge \tau^n_A}$ is uniformly integrable.
With the fact that the limit $\zeta^*_\cdot$ is almost surely continuous,
we deduce that $M^{f_i,*}_\cdot$ are continuous local martingales.
By a similar argument with the functions $f_{i,j}$,
we learn about the covariance of $M^{f_i,*}$ and $M^{f_j,*}$.

Now, by the Martingale Representation theorem~\cite[pp. 170--172]{KS},
we obtain a probability space where $\zeta^*$ and a Brownian motion $B_\cdot$ are defined,
and a matrix-valued function $\sigma^*_{i,j}$ such that
\[
   M^{f_i,*}_t = \sum_{j = 1}^d\int_0^t \sigma_{i,j}^*(\zeta^*_s) \, dB^j_s,
\]
fulfilling the requirements for a weak solution of the Martingale Problem~\cite[p. 305]{KS}.

In particular, rearranging the martingales $M^{f_i,*}$,
we get almost surely
\[
\zeta^*_t(i) =  \zeta^*_0(i) +\int_0^t b_i(\zeta^*_s) \, ds +  \sum_{j = 1}^d\int_0^t \sigma_{i,j}^*(\zeta^*_s) \, dB^j_s \quad
\forall t \geq 0, \; \forall \, 1 \leq i \leq d .
\]
If the initial condition $\zeta^n_0 \to \rho_0 $,
this shows that $\zeta^*_\cdot$ is a weak solution to the vector-valued SDE
\[
\begin{cases}
  d\zeta^*_t = b^*(\zeta^*_t)\, dt + \sigma^*(\zeta^*_t)\, dB_t\\
  \zeta^*_0 = \rho_0.
\end{cases}
\]

In our case, we use coordinates $x, y, \ldots \in V$
instead of $i \in \{\, 1, \ldots, d \,\}$.
So, the term  $b_x^*(\zeta)$ corresponds to
$\mc{L}_* \zeta(x) =  \Big[\Delta_{V,p}\zeta(x) - \beta\prt{\zeta(x)}^k\Big]$.
The matrix $\sigma$ is a square root of the quadratic variation of the local Martingales $M^{f_x,*}_t$,
and $\sigma_{x,y}(\zeta) = \delta_{x,y}\sqrt{\alpha\prt{\zeta(x)}^l}$.
This yields equation~\eqref{eq:kl}.


Now that we've presented the relation between the major concepts in our proof,
we can portray them in the following diagram:

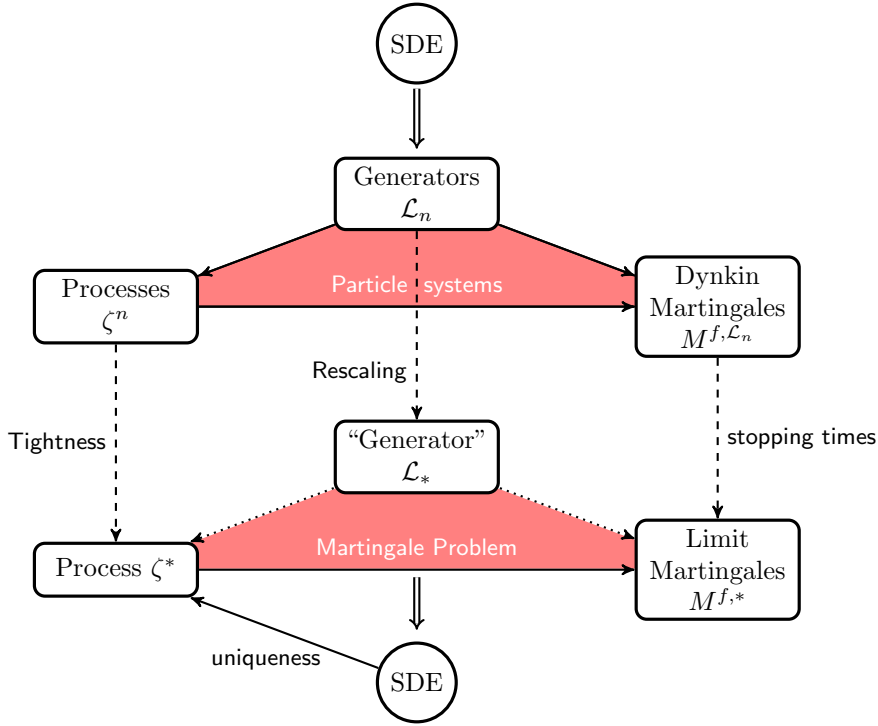
\begin{center}
\begin{tikzpicture}
[->,>=stealth',
,auto,node distance=.2\textwidth,
  thick,bolinha/.style={rectangle,rounded corners,draw=black,very thick,text centered,text width=5.5em,minimum height=2em,fill=white},
  limitf/.style={dashed},
  virtual/.style={dotted},
]

  \draw[fill=red!50] (4,6) -- (0,4.5) -- (8,4.5) -- cycle ;
  \node[bolinha] at (4,6) (0)  {Generators $\mc{L}_{n}$};
  \node[bolinha] at (0,4.5) (1)  {Processes $\zeta^{n}$} ;
  \node[bolinha] at (8,4.5) (2)  {Dynkin Martingales $M^{f,\mc{L}_n}$};

  \draw[color=white,fill=red!50] (4,2.5) -- (0,1) -- (8,1) -- cycle;
  \node[bolinha] at (4,2.5) (Ls) {``Generator'' $\mc{L}_*$};
  \node[bolinha] at (0,1) (3)  {Process $\zeta^{*}$} ;
  \node[bolinha] at (8,1) (4)  {Limit Martingales \\ $M^{f,*}$};

  \node[circle,draw=black,very thick]  at (4,-0.5) (5)  {SDE};
  \node[circle,draw=black,very thick]  at (4, 8 ) (6)  {SDE};

  \path[every node/.style={font=\sffamily\small}]
    (0) edge node  [ left]{}  (1)
    (0) edge node  [ left]{}  (2)
    (1) edge node  [above,text=white]{Particle \ systems {}}  (2)
   (0) edge[limitf] node [ left, near end]{Rescaling} (Ls)
   (1) edge[limitf] node [ left]{Tightness}  (3)
   (2) edge[limitf] node [right]{stopping times}  (4)
   (Ls) edge[virtual] node [ left]{} (3)
   (Ls) edge[virtual] node [ left]{} (4)
   (3) edge node  [above,text=white]{Martingale Problem} (4)
   (5) edge node  [ left, very near start]{uniqueness \ \ {} }(3);

  \draw[double,thick,double equal sign distance,-implies] (4,0.9) -- (4,0.2);
  \draw[double,thick,double equal sign distance,-implies] (4,7.4) -- (4,6.6);

\end{tikzpicture}
\end{center}

Section~\ref{ssec:discretemodels} is devoted to the study of the particle systems,
from the construction of the processes $\eta^n_\cdot$ to the analysis of the scaled $\zeta^n_\cdot = \frac{\eta^n_\cdot}{n}$
and their associated Dynkin martingales.

Section~\ref{sec:limitprocess} deals with the convergence of the scaled processes to the solution of the SDE.
First, we prove that the family of processes $\prt{\zeta^n_\cdot}_n$ is tight,
then we establish the conditions necessary for the limit points to solve SDE~\eqref{eq:kl},
using the equivalence of Martingale problems and SDE's.
Finally, we show that the limit point is unique.

\section{The discrete models}
\label{ssec:discretemodels}
In this Section, we complete the first part of the scheme of the proof presented in Section~\ref{sec:strategyofproof}.
We begin with a discussion on how to build general reaction-diffusion models.
Then, inspired by a fixed SDE from~\eqref{eq:kl},
we construct a family of particle systems that, after scaling, one hopes will converge to a solution to the given SDE.
With the transition rates in hand, we prove that the processes are non-exploding, which completes their construction.
Since the proof gives uniform non-explosion, it prepares for the tightness arguments.
Finally, we study the Dynkin Martingales with respect to the coordinate functions,
and introduce the notion of ``discrete coefficients'' for $L_n$.

To sum up, the purpose of this Section is to prepare the key ingredients for
the passage to the limit that takes place in Section~\ref{sec:limitprocess}.

\subsection{Setup}
\label{ssec:setup}

Fix a finite set $V$.
The points of this set, denoted by $x$, $y$, $z$, $\ldots$ are called \emph{sites}.
We denote by $\eta(x)$ the number of particles at site $x$
and by $\eta = \prt{\eta(x)}_{x \in V}$ the configuration of the system.

\medskip

For the dynamics,
there are 3 types of transition that occur independently of one another and at random times.
Particles on site $x$ may jump to site $y$
according to exponential times with rate $p(x,y)$.
Moreover, at each site $x$, a particle can be created with rate $F^+(\eta(x))$
and annihilated with rate $F^-(\eta(x))$.
We adopt the usual notation for the resulting configuration after each transition:
$\eta^{x,y}$, $\eta^{x,+}$ and $\eta^{x,-}$ respectively.

Our processes will be well-defined once we rule out explosions,
which occur when, in a finite time, an infinite number of transitions take place.
In our case, this can only happen if the total number of particles reaches $+\infty$ in a finite time.
As we'll see in subsection~\ref{sec:nonexp}, the probability of such an event is zero.

The above description corresponds to a Markov Process with denumerable state space,
see~\cite[Chapter 1]{Lannotes} for a canonical construction of such processes.
In our case, starting from each configuration $\eta$,
we obtain a probability measure $\bb{P}^\eta$;
the corresponding expectation we denote by $\bb{E}^\eta$.

At this point, we define a formal generator from pointwise convergence
\begin{equation}
\mc{L}f(\eta) = \lim_{h\downarrow 0} \frac{1}{h} \bb{E}^\eta\crt{f(X_h) - f(X_0) }\label{e:DL}.
\end{equation}
If, for a given function $f$, for every $\eta\in \bb{R}^V$ the above limit exists and
\begin{equation}
\label{eq:dynkin}
  M^{f}_t = f(\eta_t) - f(\eta_0) - \int_0^t \mc{L} f(\eta_s)\, ds.
\end{equation}
are local martingales, $f$ \emph{belongs to the domain of $\mc{L}$}.

From our dynamics, we have:
\begin{align*}
\mc{L}f(\eta) = \sum_{x,y \in V} \eta(x)p(x,y)\crt{f(\eta^{x,y}) - f(\eta)}
  & + \sum_{x \in V} F^+(\eta(x)) \crt{f(\eta^{x,+}) - f(\eta)} \\
  & + \sum_{x \in V} F^-(\eta(x)) \crt{f(\eta^{x,-}) - f(\eta)}.
\end{align*}
Transition rates can be read from the formal generator
as the factors multiplying the differences
between the transitions $f(\eta^\bullet)$ and the original state $f(\eta)$.
Note that since the particles at a given site $x$ move independently,
the transition rate for jumps from a particle from $x$ to $y$ is multiplied by $\eta(x)$.

\subsection{Construction of the models}
\label{ssec:construction}

Given parameters $\alpha,\beta,k,l$, we will construct a family of Reaction-Diffusion models on the finite set $V$,
from which we will derive a family of stochastic processes $\prt{\eta^n_\cdot = \chv{\eta^n_t; t \geq 0}}_{n \in \bb{N}}$.
The goal is to show that this family converges after scaling to the solution of \eqref{eq:kl}.

The reaction-diffusion models can be described by their formal generators $L_n$:
\begin{align}
L_nf(\eta) = \sum_{x, y \in V} \eta(x)p(x,y)\crt{f(\eta^{x,y}) - f(\eta)} & + \sum_{x \in V} F_n^{+}(\eta(x)) \crt{f(\eta^{x,+}) - f(\eta)}\nonumber \\
          & + \sum_{x \in V} F_n^{-}(\eta(x)) \crt{f(\eta^{x,-}) - f(\eta)}
\label{gen}
\end{align}
where the rates $F_n^+(\eta(x))$ and $F_n^-(\eta(x))$ now depend on $n$ and will be chosen to match the desired asymptotic evolution of the ``total mass'' of particles given by the functions
$\ID \zeta^n_t(x) = \eta^n_t(x)/n$.
\medskip

To construct $F_n^+$ and~$F_n^-$,
we will study the one-dimensional case where $V = \chv{x}$.
Here, no diffusion occurs and moreover we simplify the notation writing $\zeta^n = \eta(x)/n$.
Now, instead of changing time and space scales as in section~\ref{sec:1d},
we will choose $F^+$ and $F^-$ with more freedom,
bearing in mind the asymptotic desired limit.

Consider the evolution of the macroscopic quantity $\zeta^n$ due to birth and death of particles.
It is characterized by its infinitesimal increments, in mean and variance.
These are, respectively, given by the two following functions, as we will see in~\ref{ssec:dynkin}:
\begin{equation}
  \begin{cases}
    F_n(n\zeta^n) &= \frac{1}{n}   \crt{ F_n^{+}(n\zeta^n) -  F_n^{-}(n\zeta^n)} \\
    G_n(n\zeta^n) &= \frac{1}{n^2} \crt{ F_n^{+}(n\zeta^n) +  F_n^{-}(n\zeta^n)}
  \end{cases}
\label{relquant}
\end{equation}

With SDE~\eqref{eq:kl} in mind, we would like to set $F_n(n \zeta^n)$ and $G_n(n \zeta^n)$ to:
\begin{equation}
  \begin{cases}
    \frac{1}{n}   \crt{ F_n^{+}(n\zeta^n) -  F_n^{-}(n\zeta^n)}  &=   -\beta\prt{\zeta^n}^k \\
     \frac{1}{n^2} \crt{ F_n^{+}(n\zeta^n) +  F_n^{-}(n\zeta^n)} &= \alpha\prt{\zeta^n}^l 
  \end{cases}
\label{e:systemofequations}
\end{equation}
which is enough to determine $F_n^+$ and $F_n^-$ and therefore the generator $L_n$.
However, this would lead to
\[
2 F_n^+(n \zeta^n) = n^2 \alpha \prt{\zeta^n}^l - n\beta\prt{\zeta^n}^k
\]
which is negative for some values of $\zeta^n$. But since  transition rates must be positive we  correct this by defining:
\[
2 F_n^+(n \zeta^n) = \max\chv{n^2 \alpha \prt{\zeta^n}^l - n\beta\prt{\zeta^n}^k, 0}.
\]
Keeping the second equation of~\eqref{e:systemofequations} unchanged, we have:
\[
F_n^-(n \zeta^n) = n^2 \alpha \prt{\zeta^n}^l - F_n^+(n \zeta^n).
\]

This adjustment in $F_n^+$ forces a change to the first equation of~\eqref{e:systemofequations}, introducing an error term:
\begin{equation}
  \begin{cases}
    F_n(\zeta^n) = \frac{1}{n}   \crt{ F_n^{+}(n\zeta^n) -  F_n^{-}(n\zeta^n)} & =   -\beta\prt{\zeta^n}^k  + \text{error}_n(\zeta^n) \\
    G_n(\zeta^n) = \frac{1}{n^2} \crt{ F_n^{+}(n\zeta^n) +  F_n^{-}(n\zeta^n)} &= \alpha\prt{\zeta^n}^l 
  \end{cases}
\label{e:correctedsystemofequations}
\end{equation}
The error term appears when
$n^2 \alpha \prt{\zeta^n}^l - n\beta\prt{\zeta^n}^k < 0$.
Explicitly, when
\[
\def\bph{\vphantom{\beta}}
  \begin{cases}
   \zeta^n < \prt{\frac{\beta}{n \alpha\bph}}^{1/(l-k)} & \text{, for } k<l ;\\
   \ 1\, > \ \frac{\beta}{n \alpha\bph} & \text{, for } l=k ;\\
   \zeta^n > \prt{\frac{n \alpha}{\beta}}^{1/(k-l)} & \text{, for } k>l.\\
  \end{cases}
\]

To study our SDE, we need only the limit behaviour of the particle systems as $n\to\infty$.
Therefore, in the case $k=l$, it is enough to take $n > \beta/\alpha$ and there will be no need for correction.
When $k<l$, the error term tends uniformly (in $\zeta$) to zero as $n\to\infty$.
Finally, when $k>l$
the error term appears in a neighborhood of $+\infty$.
In this case, it is also true that the error disappears uniformly
as long as $\zeta^n$ remains in a bounded region as $n \to \infty$.

In the three cases above, from the first equation in \eqref{e:correctedsystemofequations} we have:
\begin{equation}
  \forall A > 0, \quad \sup_{\zeta \leq A} \norm{F_n(\zeta) - (- \beta\zeta^k )} \xrightarrow[n\to\infty]{} 0.
  \label{eq:errorterm}
\end{equation}

\medskip

In the general case of finitely many sites, we define the birth and death rates $F_n^+$ and $F_n^-$
by the same formula we obtained above.

\bigskip

This shows how the parameters $\alpha$, $\beta$, $k$ and~$l$ determine the discrete models (that is, the transition rates)
compatible with the convergence of the macroscopic quantities to the solution of the SDE.

Therefore, for each $\rho_0$,
fixing initial conditions $\eta^n_0$ satisfying $\zeta^n_0 \to \rho_0$,
we obtain a family 
$\prt{\prt{\eta^{n}_\cdot(x)}_{x \in V}  = \eta^{n}_\cdot}_{n \in \bb{N}}$ of processes that we will show converges to a solution to the SDE~\eqref{eq:kl}.

\subsection{Notation}
 Since the process $\eta^n_\cdot$ corresponds to a probability measure $\bb{P}_n$ on the space $D = D([0,\infty), \bb{R}^{\abs{V}})$ it will be convenient to write $\bb{P} \crt{\eta^n_ \cdot \in A} = \bb{P}_n \prt{A}$ for an event $A$ in the sigma-algebra generated by the Skorohod Topology in $D$ see ~\cite[Chapter 3, theorem 12.5]{B}.
In the same spirit, we let $\bb{E}_n$ denote the expectation with respect to  $\bb{P}_n$ and  we write, for $f:D \to \bb{R}$,
$\bb{E} \crt{f(\eta^n_\cdot)} = \bb{E}_n \crt{f}$.

\medskip

We now define the function $ S^n(\eta) : = \sum_{x \in V} \frac{\eta(x)}{n}$.
When applied to  $\eta^n_t$ we  denote it by $S^n_t =\sum_{x \in V}\zeta^n_t(x)$ which is a measure of the total mass of particles in the system for the process $\zeta^n_\cdot$ at time $t$.
This allows us to define the following stopping times:
  \begin{align*}
    \tau^n_K&:= \inf\chv{t > 0 ; S^n_t>  K}\\
    \hat{\tau}^n_0&:= \inf\chv{t > 0 ; S^n_t=0}
  \end{align*}
We will use these stopping times to have good estimates on the increment rate of the macroscopic quantities of the process
and to confine the process in bounded regions of $\bb{R}^V$ with high probability.

\subsection{Uniform boundedness  of  $\chv{S^n_\cdot}_n$}
\label{sec:nonexp}

We would like to bound the probability that
the total mass $S^n_\cdot$ attains a large value before reaching $0$,
given a bounded initial condition.
That is, we want to prove the following estimate:
\begin{lemma}
\label{l:domest}
 If $ \sup_n S^n_0 \leq C_0$ then
  \[
  \bb{P} \crt{\tau^n_K < \hat{\tau}^n_0} \leq \frac{C_0}{K}.
  \]
\end{lemma}
\begin{proof}
First, remark that $F^-_n > F^+_n$ for every $n$.
This implies that, at every site, more particles are being annihilated than created, on average.
Then couple $\eta^n$ with $\xi^n$ in such a way that $\eta^n \leq \xi^n$
and for $\xi^n$ the rate of creation of particles equals the rate of annihilation.
With this, we stochastically dominate $S^n$ by a symmetric random walk on $n^{-1}\bb{Z}$.

More precisely, consider $\prt{U_n}_{n \in \bb{N}}$ a family of independent uniform random variables,
and define $\xi^n$ according to the following rule.
Begin with $\xi^n_0 = \eta^n_0$ and whenever a particle birth or a jump occurs at $\zeta^n$ replicate the transition for $\xi^n$.
But when a particle is annihilated at time $t$ at any site $x \in V$ for the process $\eta^n_\cdot$,
we only annihilate a particle at site $x$ at time $t$ for the processes $\xi^n_\cdot$
whenever $U_{T^n(t)} > \frac{F^-_n(\eta^n(x)) - F^+(\eta^n(x))}{2\prt{F^-_n(\eta^n(x)) + F^+(\eta^n(x))}}$;
otherwise we create a particle at that site.
Here $T^n(t)$ is a variable that counts how many transitions have occurred in the process $\eta^n_\cdot$ until time $t$.

Now, define $W^n_t : = \sum_{x \in \bb{V}} \frac{\xi^n_t(x)}{n}$.
Because of the modified rates of $\xi$,
$W^n_t$ is a symmetric random walk in $n^{-1}\bb{Z}$.
Since $\xi^n_t \geq \eta^n_t$ we have that $W^n_t \geq S^n_t$ for every $t$.
Setting the corresponding stopping times
\begin{align*}
  T^n_k &= \inf \chv{t>0 \mid W^n_t > K}\\
  \hat{T}^n_0 &= \inf \chv{t>0\mid W^n_t = 0},
\end{align*}
the coupling between $\eta$ and $\xi$ implies that
$\tau^n_K > T^n_K$ and $\hat{\tau}^n_0 < \hat{T}^n_0$.
Therefore:
\[
\bb{P} \crt{\tau^n_K < \hat{\tau}^n_0} \leq \bb{P} \crt{T^n_K < \hat{T}^n_0} \leq \frac{C_0}{K}.
\]
\end{proof}

\textbf{Remark:}
The above result of uniform non-explosion implies that, almost surely,
$\lim_{K}\tau^n_K =  \infty$.
Note that every function $f: {\bb{N}}^V\to \bb{R}$
is bounded on the set $\chv{\eta\mid S^n(\eta) \leq K}$,
and this implies that the Dynkin formula~\eqref{eq:dynkin}
holds for this choice of stopping times, see~\cite[pg. 43]{LigIPS} for a proof.
Therefore, the domain of $L_n$ is the set of all functions $f:{\bb{N}}^V\to \bb{R}$.

\subsection{Local Martingales and useful computations}
\label{ssec:dynkin}
For $f: \bb{N}^V \to \bb{R}$,
the variance at time $t$ of the martingale $M^{f,L_n}_{t\wedge \tau}$ is~\cite[pg 86 Theorem 4.2.1 \textit{(v)}]{StroockVaradhan}:
\[
\bb{E}\crt{\prt{M^{f,L_n}_{t\wedge \tau}}^2} =  \bb{E}\crt{\int_0^{t\wedge\tau}(Q_n f)(\eta^n_s)\, ds}
\]
where $Q_nf(\eta) := (L_nf^2)(\eta) - 2f(\eta)L_nf(\eta)$.

\medskip

We will need to consider the expression of the local martingales
associated to the functions $f_{x,n}(\eta) = \frac{\eta(x)}{n}$ and $f_{x,y,n} = \prt{f_{x,n}\cdot f_{y,n}} (\eta)$.
We give here the corresponding values of $L_n f$ and $Q_n f$.

While $L_n f$ is given by~\eqref{gen}, $Q_n f$ is given by:
\begin{multline}
Q_nf(\eta) = \sum_{x , y \in V} \eta(x)p(x,y)\crt{f(\eta^{x,y}) - f(\eta)}^2 + \sum_{x \in V} F_n^{+}(\eta(x)) \crt{f(\eta^{x,+}) - f(\eta)}^2 \\
           + \sum_{x \in V} F_n^{-}(\eta(x)) \crt{f(\eta^{x,-}) - f(\eta)}^2
\end{multline}

From the definitions of $L_n$ and $Q_n$,
and the relations on~\eqref{e:correctedsystemofequations}, we have:
\begin{align}
  L_n(f_{x,n})(\eta)&= \Delta_{V,p} \zeta(x) -\beta \prt{\zeta(x)}^k  + \text{error}_n(\zeta(x)) \label{eq:ln} \\
  Q_n(f_{x,n})(\eta)&= \sum_{y} \frac{p(y,x)\zeta(y) + p(x,y)\zeta(x)}{n} + \alpha\prt{\zeta(x)}^\ell
  \label{eq:qn}
\end{align}
After lengthy computations, we see that $L_n f_{x,y,n}(\eta)$ can be written as:
\begin{equation}
L_n f_{x,n} (\eta) \cdot f_{y,n}(\eta) + L_n f_{y,n} (\eta) \cdot f_{x,n}(\eta)
- \prt{\frac{p(x,y)\zeta(x) +p(y,x)\zeta(y)}{n}}.
\label{e:cross_qn}
\end{equation}

\subsection{Discrete analogues of $a$ and $b$ for $L_n$}
\label{ssec:discreteanalogues}
We want to prove that the limit points of the family $\prt{\zeta^n}_n$
are solutions to a martingale problem with respect to an elliptic second-order differential operator $L_*$.
These operators have the general form:
\begin{equation}
  \label{e:diffop}
  (L_*f)(\zeta) = \sum_{x,y \in V} a^*_{x,y}(\zeta) \partial_{x,y} f(\zeta) + \sum_{x\in V} b^*_x(\zeta) \partial_xf(\zeta).
\end{equation}
To recover $a^*$ and $b^*$, we use the coordinate functions $f_x(\zeta) = \zeta(x)$:
\begin{align*}
  b^{*}_x(\zeta)     & = L_{*}f_x(\zeta)\\
  a^{*}_{x,y}(\zeta) & = L_{*}\big(f_x\cdot f_y\big) (\zeta) - f_x(\zeta) L_{*}f_y(\zeta)  - f_y(\zeta) L_{*}f_x(\zeta)\\
  a^{*}_{x,x}(\zeta) & = L_{*} f_x^2 (\zeta) - 2 f_x(\zeta) L_{*} f_x(\zeta)
\end{align*}

By analogy, we define the coefficients for $L_n$:
\begin{align*}
    b^{n}_x(\zeta)     & :=  L_{n} f_{x,n}(\eta)\\
    a^{n}_{x,y}(\zeta) & :=  L_{n} \big(f_{x,n}\cdot f_{y,n}\big) (\eta) - f_{x,n}(\eta) L_{n}f_{y,n}(\eta)  - f_{y,n}(\eta) L_{n}f_{x,n}(\eta)\\
    a^{n}_{x,x}(\zeta) & :=  L_{n} f_{x,n}^2 (\eta) - 2 f_{x,n}(\eta) L_{n} f_{x,n}(\eta) = Q_n f_{x,n}(\eta)
\end{align*}
In light of equations~(\ref{eq:ln}, \ref{eq:qn} and~\ref{e:cross_qn}), we have that:
\begin{equation}
\begin{aligned}
  b^{n}_x (\zeta)& = \Delta_{V,p} \zeta(x) - \beta\prt{\zeta(x)}^k + \text{error}_n(\zeta(x))\\
  a^{n}_{x,y}(\zeta) &= -\,\frac{p(x,y)\zeta(x) + p(y,x)\zeta(y)}{n} \\
  a^{n}_{x,x}(\zeta) & = \sum_y \frac{p(x,y)\zeta(x) + p(y,x)\zeta(y)}{n} + \alpha\prt{\zeta(x)}^\ell
\end{aligned}
\label{e:discr.coef}
\end{equation}

\section{Passing to the limit process}
\label{sec:limitprocess}
After having constructed the family $\prt{\zeta^n_\cdot}_{n\in \bb{N}}$ in the previous Section, 
we now show that it converges to the solution of the SDE~\eqref{eq:kl}.

The classical way of proving convergence is by proving tightness of the family
and uniqueness of its limit points.
The reason tightness is important is that in our setting Prohorov's Theorem~\cite[p. 51 and 138]{B} holds
and ensures that tight families are pre-compact.
This will give us limit points $\zeta^*_\cdot $ of the family $\prt{\zeta^n_\cdot}_{n\in\bb N}$.
Without limit points, we couldn't start to prove convergence.

From a convergent subsequence, we derive properties of the limit
by studying the associated coordinate martingales.
We will show that the limit point solves a well-posed Martingale Problem.
From the correspondence of the martingale problem and the SDE at hand,
we deduce uniqueness of the limit point.
This implies convergence of the whole family $\prt{\zeta^n_\cdot}_{n\in\bb N}$
to the solution of equation~\eqref{eq:kl}.

\bigskip

We write $\bb{P}^n$ to refer to the probability induced by $\zeta^n_\cdot$,
$\bb{P}^n \prt{A} = \bb{P} \crt{\zeta^n_\cdot  \in A}$.
Analogously, $\bb{E}^n$ denotes the expectation with respect to $\bb{P}^n$, so
$\bb{E}^n \crt{f(\omega)} = \bb{E} \crt{f(\zeta^n_\cdot)}$.

\subsection{Tightness}
To prove tightness for the vector-valued process $\chv{\zeta^n_\cdot}_{n \in \bb{N}}$
it suffices to verify that for every $x \in V$
the sequence of processes $\{\zeta^n_t(x) \,;\, t \in [0,\infty)\}_{n\in\bb{N}}$ in $D([0,\infty); \bb{R})$ is tight.

\begin{prop}
The sequence $\{\zeta^n_t(x) \,;\, t \in [0,\infty)\}_{n\in\bb{N}}$ satisfies Aldous's Criterion~\cite[p. 178]{B},
and therefore it is tight.
\end{prop}
\begin{proof}
We need to show that, given any $T > 0$, our sequence satisfies the following two conditions:
\begin{align*}
& i) &&
  \lim_{A \to \infty} \sup_{n }\bb{P}\crt{\sup_{t \leq T}\abs{\zeta^n_t(x)}> A}=0 \\
& ii) \ \forall \, \epsilon>0: &&
  \lim_{\delta_0 \to 0} \sup_{n \in \bb{N}} \sup_{\delta \leq \delta_0} \sup_{\tau \in \mc{T}_T} \bb{P}\prt{\abs{\zeta^n_{\tau+\delta}(x)-\zeta^n_\tau(x)}> \epsilon }= 0 \\
&&&
\text{where $\mc{T}_T = \chv{\tau \mid \text{$\tau$ is a stopping time bounded by $T$}}$\footnote{We abuse notation  $\tau + \delta = \min\chv{\tau + \delta , T}$ for convenience.}}
\end{align*}

Condition $i)$ follows from Lemma \ref{l:domest}, since
\[
  \bb{P}\crt{\sup_{t \leq T}\abs{\zeta^n_t(x)} > A} \leq \bb{P}\crt{\sup_{t \leq T} S^n_t > A} = \bb{P}\crt{\tau^n_A < \hat{\tau}^n_0}.
\]

The basic idea for condition $ii)$
is to write differences as integral expressions over a time interval of length $\delta$,
for which we will use local martingales.
Let $M^{x,n}$ denote the local martingale
associated with the function $f_{x,n}(\eta) = \frac{\eta(x)}{n}$.
Upon rewriting its defining equation~\eqref{eq:dynkin}, we get:
\[
  \zeta^n_{t}(x) = \zeta^n_0(x)
      + \prt{\int_0^{t} \Delta_{V,p} \zeta^n_s(x) -\beta\prt{\zeta^n(x)}^k  + \text{error}_n(\zeta^n) \, ds}
      + M^{x,n}_{t}.
\]
So, for $\delta \leq  \delta_0$:
\[
  \zeta^n_{\tau+\delta}(x)-\zeta^n_{\tau}(x)
  = \prt{\int_{\tau}^{\tau + \delta }\hspace{-.59cm} \Delta_{V,p} \zeta^n_s(x) -\beta\prt{\zeta^n(x)}^k\hspace{-.1cm}  + \text{error}_n(\zeta^n) \, ds}
    + \prt{M^{x,n}_{\tau + \delta} - M^{x,n}_\tau\vphantom{\int}}.
\]
Since we want a bound on the left-hand side, we will bound each term on the right by $\epsilon/2$.

We start with an estimate of the integral term.
Observe that
\begin{multline*}
  \bb{P} \crt{\;\abs{\int_\tau^{\tau + \delta} \Delta_{V,p} \zeta^n_s(x) -\beta\prt{\zeta^n}^k  + \text{error}_n(\zeta^n))\, ds}>\frac{\epsilon}{2}}\\
  \leq \bb{P} \crt{\;\abs{\int_{\tau \wedge \tau^n_A}^{(\tau + \delta) \wedge \tau^n_A} \Delta_{V,p} \zeta^n_s(x) -\beta\prt{\zeta^n}^k  + \text{error}_n(\zeta^n))\, ds}>\frac{\epsilon}{2}\;} + \bb{P} \crt{\tau^n_A< T}.
\end{multline*}
Using again Lemma \ref{l:domest}, it follows that
\[
\sup_n\bb{P} \crt{\tau^n_A< T} = \sup_n\bb{P} \crt{\sup_{t \leq T}
  S_t^n(x) > A} \xrightarrow[A \to \infty]{} 0.
\]
The stopping time $\tau_A^n$ ensures
that every dependence on $\zeta$ is uniformly bounded by $A$.
Therefore, the absolute value of the integrated term is bounded by a constant depending on $A$:
\[
  \abs{\Delta_{V,p} \zeta^n_s(x)-\beta\prt{\zeta^n}^k + \text{error}_n(\zeta^n)}
   \leq  C_1(A).
\]

Coming back to the stopped integral, by Markov's inequality:
\begin{align*}
  & \bb{P} \crt{\;\abs{\int_{\tau \wedge \tau^n_A}^{(\tau + \delta) \wedge \tau^n_A} \Delta_{V,p} \zeta^n_s(x)-\beta\prt{\zeta^n}^k  + \text{error}_n(\zeta^n)\, ds}>\frac{\epsilon}{2}\;} \\
  & \qquad \leq \frac{2}{\epsilon} \bb{E} \crt{\int_{\tau \wedge \tau^n_A}^{(\tau + \delta) \wedge \tau^n_A} \abs{\Delta_{V,p} \zeta^n_s(x)-\beta\prt{\zeta^n}^k  + \text{error}_n(\zeta^n)}\, ds} \\
  & \qquad \leq \frac{2}{\epsilon} \delta \, C_1(A)
    \leq \frac{2}{\epsilon} \delta_0 \, C_1(A) \xrightarrow[\delta_0 \to 0]{} 0.
\end{align*}

We estimate $\bb{P}\crt{\;\abs{ M^{x,n}_{\tau +\delta} - M^{x,n}_\tau}> \frac{\epsilon}{2}\;}$ by a similar stopping-time argument,
so the remaining term is
\begin{align*}
  & \bb{P}\crt{\;\abs{ M^{x,n}_{(\tau + \delta)\wedge \tau^n_A} - M^{x,n}_{\tau\wedge \tau^n_A}}> \frac{\epsilon}{2}}\\
\noalign{which by Chebychev's inequality and the martingale property of $M^{x,n}_\cdot$,}
  & \qquad \leq \frac{4}{\epsilon^2}\bb{E}\crt{\prt{M^{x,n}_{(\tau + \delta) \wedge \tau^n_A}}^2 - \prt{M^{x,n}_{\tau \wedge \tau^n_A}}^2}\\
  & \qquad  =   \frac{4}{\epsilon^2}\bb{E}\crt{\int_{\tau\wedge \tau^n_A}^{(\tau + \delta) \wedge \tau^n_A} Q_n(f_{x,n})(\eta^n_s) \, ds}\\
\noalign{and bounding $Q_n(f_{x,n})$ using again the bounds on $\zeta^n_{t\wedge \tau^n_A}$:}
  & \qquad \leq \frac{4}{\epsilon^2}\bb{E}\crt{\int_{\tau\wedge \tau^n_A}^{(\tau + \delta) \wedge \tau^n_A} C_2(A) \,ds}\\
  & \qquad \leq \frac{4}{\epsilon^2} \delta \, C_2(A)
    \leq \frac{4}{\epsilon^2} \delta_0 \, C_2(A)
    \xrightarrow[\delta_0 \to 0]{} 0 .
\end{align*}

This shows that the family $\chv{\zeta^n_\cdot(x)}_n$ satisfies condition $ii)$ of Aldous's Criterion,
and concludes the proof of its tightness.
\end{proof}

\subsection{Characterization of the limit}

Now that we proved that the sequence $\prt{\zeta^n_\cdot}_n$ --- or, equivalently, $\prt{\bb{P}^n}_n$ --- is tight, it remains to characterize its limit points.
For this, we will show that:
\begin{itemize}
\item  the limit points are continuous,
\item  the limit points solve the martingale problem for some second-order elliptic differential operator $L_*$, and that
\item  there is a unique solution to the martingale problem for $L_*$.
\end{itemize}

\subsubsection{Continuity of  paths of the limit process}
We claim that any limit measure of the tight family  $\prt{\bb{P}^n}_{n \in \bb{N}}$  gives measure $1$ to continuous trajectories.
This follows from an analysis of the jump function:
\begin{align*}
  J: D([0,T], E)& \to \bb{R}\\
x & \mapsto \sup_{t \in (0,T]} d_E(x(t -), x(t)).
\end{align*}
Since every $x \in D$ is right-continuous, the condition $J(x) = 0$ implies that $x$ is a left-continuous path in $[0,T]$
and therefore $x$ is continuous in $[0,T]$.

Let $\bb{P}^*$ be a limit point of the sequence $\bb{P}^n$.
That is, for some subsequence $\bb{P}^{n'}$,
we have $\bb{P}^{n'} \xrightarrow[n'\to \infty]{J_1} \bb{P}^*$.
We claim that $\bb{P}^*\crt{J(x) = 0} = 1.$
The key observation to  prove this is to note that $J$ is continuous with respect to the $J_1$-Skorohod topology~\cite[p. 125]{B}.
This means that
\[
  \crt{J(x) > a} = \chv{x\mid J(x)>a} = J^{-1}\prt{(a,\infty)}
\]
is an open set.
Since jumps of $\zeta^n$ have magnitude $1/n$,
for $n > K$, $\bb{P}_n \crt{J(x) > \frac{1}{K}} = 0$
so by the Portemanteau Theorem~\cite[p. 16]{B}:
\[
\bb{P}^*\crt{J(x)> \frac{1}{K}} \leq \liminf_{n'} \bb{P}^{n'} \crt{J(x)> \frac{1}{K}} = 0.
\]

This implies that
\[
\bb{P}^*\crt{J(x) > 0}
  = \bb{P}^*\prt{\cup_{K\in \bb{N}} \crt{J(x) > \frac{1}{K} }}
  \leq \sum_{K\in \bb{N}} \bb{P}^*\crt{J(x)> \frac{1}{K}} = 0
\]
which is the same as $\bb{P}^* \crt{J(x) = 0} = 1$.

\subsubsection{Martingale problem}

In this section, we construct a second-order differential operator $L_*$
from the limits of the discrete coefficients of $L_n$ given in equation~\eqref{e:discr.coef},
and prove that any limit point of $\bb{P}^n$ is a solution to the martingale problem associated with it.
Remembering that the error term vanishes uniformly as $n\to\infty$, we define:
\begin{align*}
  b^{*}_x(\zeta)& = \Delta_{V,p} \zeta(x) - \beta\prt{\zeta(x)}^k\\
  a^{*}_{x,y}(\zeta) &= 0\\
  a^{*}_{x,x} (\zeta)& = \alpha\prt{\zeta(x)}^\ell,
\end{align*}
to be the coefficients of the differential operator $L_*$ as in~\eqref{e:diffop}.

\medskip

To prove that limit points of $\bb{P}^n$ solve the martingale problem associated with $L^*$,
we don't need to verify that $\big\{\, M^{f,L_*}_t, \mc{F}_t; t \geq 0 \,\big\}$
is a continuous local martingale for every $f \in C^2( \bb{R}^d)$.
By~\cite[p. 318]{KS}, we only need to show that:
\begin{prop}
\begin{align*}
  M^{x,L_*}_t   & := \zeta^*_t(x) - \zeta^*_0(x) - \int_0^t b^*_x(\zeta^*_s)\, ds\\
  M^{x,y,L_*}_t & := \zeta^*_t(x)\zeta^*_t(y) - \zeta^*_0(x)\zeta^*_0(y) \\
              & \qquad\qquad - \int_0^t b^*_x(\zeta^*_s)\zeta^*_s(y) + b^*_y(\zeta^*_s)\zeta^*_s(x) + a^*_{x,y}(\zeta^*_s) \, ds
\end{align*}
are continuous local martingales.
\end{prop}

\begin{proof}
Our proof uses Skorohod's Representation Theorem~\cite[pg 70]{B},
to consider a convenient $\Omega$ where
$\zeta^n_\cdot(\omega) \xrightarrow[n \to \infty]{J_1} \zeta^*_\cdot(\omega)$ for every $\omega \in \Omega$.
This is possible since $\zeta^n_\cdot \xrightarrow[n \to \infty]{J_1} \zeta^*_\cdot$ in law.

Then, the proof splits into the following steps:
\begin{enumerate}
\item Define $\tau^{*,n}_A$ a family of stopping times by 
\begin{equation}
  \tau^{*,n}_A := \inf\chv{t > 0 ; S^n_t> A , S^*_t \geq A}\label{e:stop}
\end{equation}
and prove that 
$ \tau^{*,n}_A \to \tau^*_A := \inf\chv{t >0 ; S^*_t \geq A}.$

\item Prove that the stopped process converge to the correct stopped limit:
\[
  \zeta^n_{\cdot \wedge \tau^{*,n}_A} \xrightarrow[n \to \infty]{J_1} \zeta^*_{\cdot \wedge \tau^*_A}.
\]

\item Show that the stopped local martingales also converge appropriately:
\[
M^{x,L_n}_{t\wedge \tau^{*,n}_A} \xrightarrow[n \to \infty]{J_1} M^{x,L_*}_{t\wedge \tau^*_A}
  \qquad \text{and} \qquad
M^{x,y,L_n}_{t\wedge \tau^{*,n}_A} \xrightarrow[n \to \infty]{J_1} M^{x,y,L_*}_{t\wedge \tau^*_A}.
\]

\item Conclude that the limit processes $M^*_{\cdot \wedge \tau^*_A}$ above are actually martingales, and therefore the non-stopped version are local martingales.
\end{enumerate}

\medskip
\textbf{First step: prove that $\tau^{*,n}_A(\omega) \to \tau^*_A(\omega)$ for every $\omega$.}

By definition, $\tau^{*,n}_A(\omega) \leq \tau^*_A(\omega)$,
so it remains to see that if $q < \tau^*_A(\omega)$ then there is $N(\omega)$ such that for $n > N(\omega)$ 
\[
q <\tau^{*,n}_A(\omega).
\]
Since $q < \tau^*_A(\omega)$, for every $s \leq q$ we have $S^*_s(\omega) < A$.
By the continuity of $S^*_s(\omega)$, there is $\epsilon(\omega)$ such that $S^*_s(\omega) < A - \epsilon(\omega)$.
Since $\zeta^n_\cdot(\omega)$ converges to $\zeta^*_\cdot(\omega)$ in the $J_1$ topology,
and $t \mapsto \zeta^*_t(\omega)$ is continuous,
$\zeta^n_\cdot(\omega)$ converges uniformly in bounded intervals to $\zeta^*_\cdot(\omega)$.
This implies that there is an $N(\omega)$ such that for $n > N(\omega)$,
$\sup_{s \leq q} \norm{S^n_s(\omega) - S^*_s(\omega)} < \epsilon(\omega)$. Then:
\[
 S^n_s(\omega)  = S^n_s(\omega) - S^*_s(\omega) + S^*_s(\omega) <  A 
\]
which implies that $\tau^{*,n}_A(\omega) \ge q$.
Therefore $\tau^{*,n}_A(\omega) \to \tau^*_A(\omega)$.

\medskip
\textbf{Second step: $\zeta^n_{\cdot \wedge \tau^{*,n}_A} \xrightarrow[n \to \infty]{J_1} \zeta^*_{\cdot \wedge \tau^*_A}$.}

Fix $\epsilon > 0$.
From the uniform convergence of $\zeta^n_\cdot(\omega)$ to $\zeta^*_\cdot(\omega)$ on $[0,T]$,
we can choose $N_1(\omega)$ such that
\[ 
  \sup_{t \leq T}\norm{\zeta^n_{t}(\omega) - \zeta^*_{t}(\omega)} < \epsilon \quad \text{ for all $n > N_1(\omega)$.}
\]
From the uniform continuity of $\zeta^*_\cdot(\omega)$ on $[0,T]$,
there is $\delta(\omega) > 0$ such that
\[
  \text{for } \abs{t - t'} < \delta(\omega), \quad \norm{\zeta^*_t(\omega) - \zeta^*_{t'}(\omega)} < \epsilon.
\]
From the previous step, we choose $N_2(\omega)$ such that $(T \wedge \tau^*_A) - (T \wedge \tau^{*,n}_A) < \delta(\omega)$.

Recalling that $\tau^{*,n}_A(\omega) \leq \tau^*_A(\omega)$, we have two cases to consider:
\begin{itemize}
\item if $t \leq \tau^{*,n}_A(\omega)$, for $n > N_1(\omega)$:
\[
\norm{ \zeta^n_{t \wedge \tau^{*,n}_A} - \zeta^*_{t \wedge \tau^*_A}}
  = \big\lVert \zeta^n_{t} - \zeta^*_{t} \big\rVert < \epsilon \ ;
\]

\item if $t > \tau^{*,n}_A(\omega)$, for $n > N_1(\omega)$ and $n > N_2(\omega)$:
\[
\norm{ \zeta^n_{t \wedge \tau^{*,n}_A} - \zeta^*_{t \wedge \tau^*_A}}
  = \norm{ \zeta^n_{\tau^{*,n}_A} - \zeta^*_{t \wedge \tau^*_A}}
  \leq  \norm{ \zeta^n_{\tau^{*,n}_A} - \zeta^*_{\tau^{*,n}_A}} + \norm{ \zeta^*_{\tau^{*,n}_A} - \zeta^*_{t \wedge \tau^{*,n}_A}} < 2 \epsilon.
\]
\end{itemize}

This proves uniform convergence on $t \in [0,T]$ of $\zeta^n_{t \wedge \tau^{*,n}_A}$ to $\zeta^*_{t \wedge \tau^*_A}$.

\medskip
\textbf{Third step:}

We are going to show that
\[
M^{x,L_n}_{t\wedge \tau^{*,n}_A} = \zeta^n_{t\wedge \tau^{*,n}_A}(x) - \zeta^n_0(x) - \int_0^{t\wedge \tau^{*,n}_A} b^n_x(\zeta^n_s)\, ds
\]
converges in the $J_1$ topology to
\[
M^{x,L_*}_{t\wedge \tau^*_A} = \zeta^*_{t\wedge \tau^*_A}(x) - \zeta^*_0(x) - \int_0^{t\wedge \tau^*_A} b^*_x(\zeta^n_s)\, ds.
\]
From the previous step, we only need to prove the convergence for the third term,
that is:
\[
 \int_0^{t\wedge \tau^{*,n}_A} b^n_x(\zeta^n_s)\, ds \xrightarrow[n \to \infty]{J_1} \int_0^{t\wedge \tau^*_A} b^*_x(\zeta^*_s)\, ds.
\]

Given $\epsilon > 0$, we can find $N$ (remember this depends on $\omega$) such that, for every $n > N$:
\begin{itemize}
\item $\sup_{s \leq \tau^{*,n}_A} \abs{b^n_x(\zeta^n_s) - b^*_x(\zeta^n_s)} < \frac{\epsilon}{T}$,
  because $\zeta^n_s(x) \leq A$ and so the error term vanishes, see~\eqref{eq:errorterm};
\item $\sup_{s \leq \tau^{*,n}_A} \abs{b^*_x(\zeta^n_s) - b^*_x(\zeta^*_s)} < \frac{\epsilon}{T}$,
  since $b^*_x$ is uniformly continuous on $[0,A]^V$ and $\sup_{s \leq T} \abs{\zeta^n_s - \zeta^*_s} \to 0$;
\item $\abs{\tau^*_A -\tau^{*,n}_A} < \frac{\epsilon}{B_A}$ where $B_A = \sup_{\zeta \in [0,A]^V} \abs{b^*_x(\zeta)}$, by step 1.
\end{itemize}

Putting these inequalities together:
\begin{align*}
  & \abs{ \int_0^{t\wedge \tau^{*,n}_A} b^n_x(\zeta^n_s)\, ds - \int_0^{t\wedge \tau^*_A} b^*_x(\zeta^*_s)\, ds} \\
  & \qquad \leq \int_0^{t\wedge \tau^{*,n}_A}\abs{b^n_x(\zeta^n_s) - b^*_x(\zeta^*_s)} \, ds + \int_{t\wedge \tau^{*,n}_A}^{t\wedge \tau^*_A}\abs{b^*_x(\zeta^*_s)} \, ds \\
  & \qquad < T \, 2 \frac{\epsilon}{T} + \frac{\epsilon}{B_A} B_A  = 3 \epsilon
\end{align*}
for all $n > N$, which concludes this step.

\medskip
\textbf{Fourth step:}

Since $\tau^{*,n}_A$ is a stopping time, $M^{x,L_n}_{t\wedge \tau^{*,n}_A}$ is martingale, and by the definition of $\tau^{*,n}_A$, uniformly bounded in $n$.
This implies that it is a sequence of uniformly integrable martingales and therefore its limit is also a martingale.
This proves step 4 and because $\tau^*_A \uparrow \infty$ when $A \uparrow \infty$, we conclude that $M^{x,L_*}_{t}$  is a local martingale.

\bigskip

The proof that $M^{x,y,L_*}_{\cdot}$ is a local martingale follows from an analogous reasoning:
We use again Skorohod's Representation Theorem, and step~1 is unchanged. 
The uniform convergence of $\zeta^n_{t \wedge \tau^{*,n}_A}(x)\zeta^n_{t \wedge \tau^{*,n}_A}(y)$
to $\zeta^*_{t \wedge \tau^*_A}(x)\zeta^*_{t \wedge \tau^*_A}(y)$ follows from what we proved in step 2 above. 
The convergence of $M^{x,y,L_n}_{\cdot \wedge\tau^{*,n}_A}$ to $M^{x,y,L_*}_{\cdot\wedge \tau^*_A}$ (step 3) follows from the following limits:
\begin{itemize}
\item $\sup_{s \leq T} \abs{\zeta^n_s - \zeta^*_s} \to 0$
\item $\abs{\tau^*_A -\tau^{*,n}_A} < \frac{\epsilon}{A}$
\item $\abs{b^n_x(\zeta) - b^*_x(\zeta)} \to 0$
\item $\abs{a^n_x(\zeta) - a^*_x(\zeta)} \to 0$
\end{itemize}
which imply that
\[
\int_0^{t\wedge \tau^{*,n}_A} b^n_x(\zeta^n_s)\zeta^n_s(y) + b^n_y(\zeta^n_s)\zeta^n_s(x) + a^n_{x,y}(\zeta^n_s) \, ds
\]
converges uniformly (in $t \leq T$) to 
\[
\int_0^{t\wedge \tau^*_A} b^*_x(\zeta^*_s)\zeta^*_s(y) + b^*_y(\zeta^*_s)\zeta^*_s(x) + a^*_{x,y}(\zeta^*_s) \, ds.
\]
We conclude the proof (step 4) by noting that $M^{x,y,L_*}_{\cdot\wedge \tau^*_A}$ 
is the limit of a family of uniformly integrable martingales (in fact the martingales are here again uniformly bounded)
and this implies that $M^{x,y,L_*}_{\cdot\wedge \tau^*_A}$ is a martingale.
\end{proof}

\subsubsection{Uniqueness}

To conclude that the tight sequence of processes $\chv{\zeta^n_\cdot}$ converges,
it suffices to prove that it has a unique limit point.
This is the final step needed for the characterization of the limits of the particle systems associated with $L_n$.

In our case, uniqueness of solutions to the martingale problem follows from
pathwise uniqueness of solutions to the correspondent SDE;
either because both coefficients $b$ and~$\sigma$ satisfy locally Lipschitz conditions,
or because the dispersion coefficients $\sigma$ vanish as $\sqrt{x}$ near the boundary
while the drift coefficients $b$ remain locally Lipschitz.
This is a limit case in the sense that if the dispersion matrix vanished with rate $x^{\alpha}$ with $\alpha < \frac{1}{2}$
we would no longer have uniqueness.

The correspondence between solutions to martingale problems and solution to SDE's is given in~\cite[Corollaries 4.8 and 4.9, p. 317]{KS}.
Existence and uniqueness of solutions $\crt{(X,W), \prt{\Omega, \mc{F}, \bb{P}}, \mc{F}_t}$ in the sense of probability law to an SDE with a fixed but arbitrary initial distribution
\[
\bb{P} \crt{X_0 \in \Gamma} = \mu(\Gamma)
\]
is equivalent to existence and uniqueness of solutions $P$ to the corresponding martingale problem with the initial condition
\[
P\crt{y \in C([0,\infty) , \bb{R}^d \mid y(0)\in \Gamma} = \mu(\Gamma) \; \Gamma \in \mc{B}(\bb{R}^d).
\]
The two solutions are related by $\bb{P}(X \in A) = P(A)$, that is, the solution to the martingale problem is the law induced by the weak solution to the SDE.

Also, uniqueness in law follows from pathwise uniqueness~\cite[p. 301 and 331]{YW,KS}.
So, we only need to verify that pathwise uniqueness holds for our equations.
We need to treat two cases separately.

\medskip

For the case $\ell \geq 2$, since both $b$ and $\sigma$ are $C^1$,
we can apply
\begin{teoutro}[\protect{\cite[p.287]{KS}}]
  Suppose that the coefficients $b(t,\zeta)$ and $\sigma(t,\zeta)$ are locally Lipschitz continuous in the space variable; i.e., for every $n \geq 1$ there exists a constant $K_n >0$ such that for every $t \geq 0, \norm{\zeta} \leq n$ and $\norm{\tilde{\zeta}} \leq n$:
\[
\norm{b(t,\zeta) - b(t,\tilde{\zeta})} + \norm{\sigma(t,\zeta) - \sigma(t,\tilde{\zeta})} \leq K_n \norm{\zeta-\tilde{\zeta}}.
\]
Then pathwise uniqueness holds for $(b,\sigma)$, that is, for the vector valued SDE
\[
d\zeta_t = b(t,\zeta) dt + \sigma(t,\zeta) \, dB_t.
\]
\end{teoutro}

\medskip

For $\ell = 1$ we note that the local Lipschitz condition is not true for $\sigma$,
due to the square-root behavior near the boundary $\zeta(x) = 0$.
Labeling the sites $V$ of our graph by $\chv{x_1,\ldots, x_n}$,
we observe that the associated dispersion matrix $\sigma$ fits the particular form required by the following criterion,
which will again give pathwise uniqueness.
\begin{teoutro}[\protect{\cite[Theorem 1]{YW}}] Let
\[
d\zeta_t = \sigma(\zeta_t)\, dB_t + b(\zeta_t)\, dt
\]
where the coefficients
\[ 
\sigma(\zeta) = \begin{bmatrix}\sigma_1(\zeta(x_1)) &&& 0 \\ & \sigma_2(\zeta(x_2)) && \\ && \ddots & \\ 0 &&&\sigma_n(\zeta(x_n))\end{bmatrix}
\quad \text{and} \quad
b(\zeta) = \begin{bmatrix}b_1(\zeta) \\ b_2(\zeta) \\ \vdots \\ b_n(\zeta)\end{bmatrix}
\]
are such that
\begin{itemize}
\item[(i)] there is a positive increasing function $\rho$ on $(0,\infty)$
such that  for every $\epsilon > 0$ $\int_0^\epsilon \frac{1}{\rho^2(u)}\, du = + \infty$, and such that
\[
\abs{\sigma_i(u) - \sigma_i(v)} \leq \rho(\abs{u - v}), \quad \forall \, u,v \in \bb{R}, \quad 1 \leq i \leq n
\]
\item[(ii)] for every $n \geq 1$ there exists a constant $K_n >0$ such that for every $t \geq 0, \norm{\zeta} \leq n$
and $\lVert \tilde{\zeta} \rVert \leq n$:
\[
\norm{b_i(\zeta) - b_i(\tilde{\zeta})}  \leq K_n \norm{\zeta-\tilde{\zeta}}.
\]
Then pathwise uniqueness holds.
\end{itemize}
\end{teoutro}

Therefore, the family $\chv{\zeta^n_t;t\geq 0}_{n \in \bb{N}}$ obtained from the discrete models corresponding to the parameters $k,\ell,\beta,\alpha$ is tight and converges to the probability measure induced by the unique solution of  \eqref{eq:kl}. This proves Theorem \ref{T:1}.

\section{Open problems}
Some generalizations of the results presented here are interesting to consider.
For instance, the case where the average net increase due to birth and death of particles is positive
requires dealing with explosion times.
Also, one would like to consider the case when $V$ has infinitely many sites.
\section{Acknowledgments}

C. Costa would like to thank CNPq for the funding of his Ph.D. thesis when this work was developped
and IMPA and ULeiden for their support in finishing the manuscript

\begin{thebibliography}{10}

\bibitem{AT}
Ludwig Arnold and Marie~Theodosopulu.
\newblock deterministic limit of the stochastic model of chemical reactions
  with diffusion.
\newblock {\em Advances in Applied Probability}, 12(2):367--379, 1980.

\bibitem{B}
Patrick Billingsley.
\newblock {\em Convergence of probability measures}.
\newblock Wiley series in probability and statistics. Wiley-Interscience, 2ed
  edition, 1999.

\bibitem{Bl92}
Douglas Blount.
\newblock Law of large numbers in the supremum norm for a chemical reaction
  with diffusion.
\newblock {\em The annals of applied probability}, 2(1):131--141, 1992.

\bibitem{BorCor}
Alexei~{Borodin} and Ivan~{Corwin}.
\newblock {Moments and Lyapunov exponents for the parabolic Anderson model}.
\newblock {\em ArXiv e-prints}, November 2012.

\bibitem{CanCos}
Robert~Stephen {Cantrell} and Chris {Cosner}.
\newblock {\em Spatial Ecology via Reaction-Diffusion Equations}.
\newblock John Wiley \& Sons, November 2003.

\bibitem{Cos}
Conrado Costa.
\newblock From reaction-diffusion models to the study of stochastic
  differential equations.
\newblock {\em IMPA e-prints}, 2016.

\bibitem{ErhdHolMai}
Dirk~{Erhard}, Frank~{den Hollander}, and Grégory~{Maillard}.
\newblock {Parabolic Anderson Model in a Dynamic Random Environment: Random
  Conductances}.
\newblock {\em Mathematical Physics, Analysis and Geometry}, 19:5, June 2016.

\bibitem{FarLanTsu}
Jonathan~{Farfan}, Claudio~{Landim}, and Kenkichi~{Tsunoda}.
\newblock {Static large deviations for a reaction-diffusion model}.
\newblock {\em ArXiv e-prints}, June 2016.

\bibitem{KS}
Ioannis Karatzas and Steven~E. Shreve.
\newblock {\em Brownian Motion and Stochastic Calculus}.
\newblock Graduate Texts in Mathematics. Springer Verlag, 2ed edition, 1991.

\bibitem{KL}
Claude Kipnis and Claudio Landim.
\newblock {\em Scaling Limits of Interacting Particle Systems}.
\newblock Springer, 1999.

\bibitem{KomLanOll}
Tomasz Komorowski, Claudio Landim, and Stefano Olla.
\newblock {\em Fluctuations in Markov processes: time symmetry and martingale
  approximation}, volume 345.
\newblock Springer Science \& Business Media, 2012.

\bibitem{LanTsu}
Claudio~{Landim} and Kenkichi.~{Tsunoda}.
\newblock {Hydrostatics and dynamical large deviations for a reaction-diffusion
  Model}.
\newblock {\em ArXiv e-prints}, August 2015.

\bibitem{Lannotes}
Claudio Landim.
\newblock Scaling limit of metastable markov chains, September 2013.

\bibitem{LeGall}
Jean-François {Le Gall}.
\newblock {Bessel processes, the Brownian snake and super-Brownian motion}.
\newblock {\em ArXiv e-prints}, July 2014.

\bibitem{LigIPS}
Thomas~Milton Liggett.
\newblock {\em Interacting Particle Systems}.
\newblock Springer Berlin Heidelberg, 2005.

\bibitem{Lig}
Thomas~Milton Liggett.
\newblock {\em Continuous Time Markov Processes}, volume 113 of {\em Graduate
  Texts in Mathematics}.
\newblock The American Mathematical Society, 1999.

\bibitem{PanTalWhi}
G.~{Pang}, R.~{Talreja}, and W.~{Whitt}.
\newblock {Martingale proofs of many-server heavy-traffic limits for Markovian
  queues}.
\newblock {\em ArXiv e-prints}, December 2007.

\bibitem{StroockVaradhan}
Daniel~W. Stroock and S.~R.~Srinivasa Varadhan.
\newblock {\em Multidimensional diffusion processes}.
\newblock Springer-Verlag Inc, Berlin; New York, 2006.

\bibitem{V}
Vitaly Volpert.
\newblock {\em Elliptic Partial Differential Equations : Volume 2:
  Reaction-Diffusion Equations}.
\newblock Monographs in Mathematics. Springer Basel, 2014.

\bibitem{YW}
Toshio Yamada and Shinzo Watanabe.
\newblock On the uniqueness of solutions of stochastic differential equations.
\newblock {\em J. Math. Kyoto Univ.}, 11(1):155--167, 1971.
\end{thebibliography}

\end{document}